\documentclass[11pt,reqno]{amsart}

\usepackage{newtxtext}
\usepackage{newtxmath}

\usepackage{cite}

\usepackage{url}
\usepackage{enumitem}

\theoremstyle{plain}
\newtheorem*{theorem*}{Theorem}
\newtheorem*{mainlemma*}{Main Lemma}
\newtheorem{lemma}{Lemma}

\theoremstyle{remark}
\newtheorem*{remark*}{Remark}

\newcommand{\aSup}[1]{\ensuremath{\mathop{\underset{#1}{\sup}}}}

\newcommand{\mEQx}[2]{\ensuremath{\stackrel{\text{#1}}{#2}}}

\newcommand{\nEps}{\varepsilon}

\newcommand{\bDiv}{\ensuremath{/}}

\newcommand{\nP}{\ensuremath{P}}
\newcommand{\nM}{\ensuremath{\varrho}}

\newcommand{\aH}[1]{\ensuremath{{#1}^{\star}}}
\newcommand{\nPh}{\ensuremath{\aH{\nP}}}
\newcommand{\nMh}{\ensuremath{\aH{\nM}}}

\newcommand{\nE}{\ensuremath{{\mathbf E}}}
\newcommand{\nEh}{\ensuremath{\aH{\nE}}}
\newcommand{\aE}[1]{\ensuremath{\nE_{#1}}}
\newcommand{\aEh}[1]{\ensuremath{\nEh_{#1}}}

\newcommand{\nU}{u}

\newcommand{\aSeqb}[1]{\ensuremath{\left\{#1\right\}}}
\newcommand{\aSeqilu}[4]{\ensuremath{\aSeqb{#1}_{{#2}={#3}}^{#4}}}
\newcommand{\aSeq}[4]{\ensuremath{\aSeqb{#1}}}

\newcommand{\aLim}[1]{\ensuremath{\mathrel{\rightarrow}}}

\DeclareMathOperator{\nLimsup}{\overline{Lim}}
\DeclareMathOperator{\nLiminf}{\underline{Lim}}

\usepackage[pdftex,unicode,bookmarksnumbered]{hyperref} 

\hypersetup{
 pdftitle={On completeness of Hausdorff hyperspaces},
 pdfauthor={Ján Komara},
 pdfsubject={Hyperspaces in general topology},
 pdfkeywords={
  Hausdorff hyperspaces{,} 
  metric spaces{,} 
  completeness{,} 
  fractals
 },
}

\hypersetup{
 colorlinks=true,
 allcolors=[rgb]{0,0.5,0.5},
 urlcolor=[rgb]{0.58,0.0,0.83},
 breaklinks=true,
}

\newcommand{\DOI}[1]{\href{https://doi.org/#1}{#1}}

\begin{document}     


\title{On completeness of Hausdorff hyperspaces}

\author{Ján Komara}

\address{
Department of Applied Informatics\\
Faculty of Mathematics, Physics and Informatics\\
Comenius University\\
Mlynská dolina F1\\ 
842 48 Bratislava\\ 
Slovakia
}

\email{komara@fmph.uniba.sk}

\thanks{Thanks to Aneta Barnes for proofreading this manuscript.}



\begin{abstract}
The Hausdorff hyperspace of a metric space 
consists of all its non-empty bounded closed sets
and it is equipped with the Pompeiu--Hausdorff set distance.
We present a simpler novel proof that 
the Hausdorff hyperspace of a complete space is complete as well.
The 
\hyperlink{cl:ml}{Main Lemma} 
is crucial in this demonstration
and though it uses an induction argument---%
the only one in our completeness proof---%
it is stated purely in terms of neighborhoods.
\end{abstract}

\maketitle


\section*{Introduction}

We focus on the following basic problem from the theory of metric spaces
\cite{bib:Hausdorff,bib:Hahn,bib:Cech,bib:Barnsley,bib:Henrikson,bib:Edgar}:
\begin{quote}
Let $\nP$ be a metric space with a distance function $\nM$ and 
$\nPh$ its Hausdorff hyperspace consisting 
of all non-empty bounded closed subsets of $\nP$
with metric $\nMh$ defined by
\begin{gather*}
\nMh(A,B) = 
\max   
\{ 
\aSup{x \in A} \nM(x,B), \aSup{y \in B} \nM(y,A) 
\}
.
\end{gather*}
The question is: Which metric and topological properties 
does the Hausdorff hyperspace $\nPh$ inherit from its base space $\nP$?
\end{quote}
The current interest in Hausdorff hyperspaces 
is driven by various applications in the theory of fractals
to provide a natural framework for the fractal geometry
\cite{bib:Barnsley,bib:Edgar}.
One of the most useful properties of a metric space is its completeness. 
According to 
\cite[Chap.~II]{bib:Barnsley},
complete Hausdorff hyperspace is ``the place where fractals live''.

The notion of set distance was introduced by Pompeiu
\cite[p.~17--18]{bib:Pompeiu} 
and later studied and partially modified by Hausdorff 
\cite[Chap.~VIII, §6]{bib:Hausdorff}---%
see 
\cite{bib:BirsanTiba} 
for a lengthy discussion on this topic.
The completeness result is due to Hahn 
\cite[18.10]{bib:Hahn}.
It has been reproved numerously---for the modern exposition, see, for example,
\cite[Thm.~7.1]{bib:Barnsley}, \cite[Thm.~3.3]{bib:Henrikson}, and \cite[Thm.~2.5.3]{bib:Edgar}. 
All these demonstrations are, as a rule, very technical and far from insightful.
Thus, more advanced improvements of this phenomenon should be considered.

We present a new, more elementary proof that Hausdorff hyperspaces inherit completeness. 
It is based on the ideas from the author's student project \cite{bib:Komara},
which were in turn inspired by the proof of a similar result
\cite[Thm.~17.6.5]{bib:Cech}
for hyperspaces consisting of non-empty compact point sets.


\section*{Prerequisites}

The reader is expected to be familiar with the basic theory of metric spaces.
We closely follow the notation and terminology of 
\cite{bib:Cech}
without further explanation.
We will use the letters $x,y,z$ for points of metric spaces,
letters $A,B$ for point sets,
letters $i,j,k,m,n$ for natural numbers $1,2,3,\ldots$,
and letters $b,\nEps$ for real numbers.

Following 
\cite[17.6]{bib:Cech}, 
we denote by $\nU(A,B)$ and $\nU(B,A)$:
\[
\nU(A,B) = \aSup{x \in A} \nM(x,B)
,
\qquad
\nU(B,A) = \aSup{y \in B} \nM(y,A)
,
\]
the distances from the set $A$ to the set $B$ and 
from the set $B$ to the set $A$, respectively.
With this notation, we can simply note down that
\[
\nMh(A,B) = \max \{ \nU(A,B), \nU(B,A) \}
.
\]
The number $\nMh(A,B)$ represents 
the reciprocal distance between the sets $A$ and $B$.

Our completeness result is based on the following characterization 
of limits in the hyperspace $\nPh$:
{\it
If the set $A$ is the limit 
of the sequence \aSeqilu{A_n}{n}{1}{\infty} in \nPh, 
then
}
\begin{align}
A = \bigcap_{n=1}^{\infty} \overline{\bigcup_{i=n}^{\infty} A_i}
.
\label{eq:limit}
\end{align}
The overbar denotes the topological closure operator on $\nP$.
This can be proved by arguments similar to those in
\cite[17.6.1--17.6.3]{bib:Cech}.


\section*{Completeness}

The next theorem contains the main result of this note.


\begin{theorem*}

If $\nP$ is a complete metric space, 
then so is its Hausdorff hyperspace $\nPh$.

\end{theorem*}

\begin{proof}
Let \aSeq{A_n}{n}{1}{\infty} be a Cauchy sequence in \nPh.
We claim that the point set $A$ defined by \eqref{eq:limit} 
is the limit of the sequence \aSeq{A_n}{n}{1}{\infty} in \nPh,
in symbols $A_n \aLim{n} A$ in $\nPh$,
or equivalently that 
$A \in \nPh$, $\nU(A,A_n) \aLim{n} 0$, and $\nU(A_n,A) \aLim{n} 0$.
The proof is thus naturally divided into three steps,
stated below as Lemmas \ref{cl:ph}, \ref{cl:uaan}, and \ref{cl:uana}.

The following auxiliary lemma is significant in their justification.
It asserts that  if there is an open ball in $\nP$ such that
all but a finite number of the point sets of the Cauchy sequence \aSeq{A_n}{n}{1}{\infty}
have an element within the ball,
then the distance of its center to the set $A$ 
is less than or equal  to its radius.
\renewcommand{\qedsymbol}{}
\end{proof}


\begin{mainlemma*}
\hypertarget{cl:ml}{}

Suppose that for a point $x$ of $\nP$
there is an $\nEps > 0$ and a number $m$ such that 
$\nM(x,A_i) < \nEps$ for every index $i > m$.
Then $\nM(x,A) \leq \nEps$.

\end{mainlemma*}

\begin{proof}

Fix $b > 1$.
Since \aSeq{A_n}{n}{1}{\infty} is a Cauchy sequence, 
a simple induction argument shows that 
there is a subsequence \aSeqilu{A_{n_i}}{i}{1}{\infty} of \aSeq{A_n}{n}{1}{\infty}
starting with $n_1 > m$ such that
\begin{align*}
(\forall j > n_i) (\forall z \in A_{n_i}) (\nM(z,A_j) < \nEps \bDiv b^{i})
\qquad
\text{($i = 1,2,3,\ldots$)}
.
\end{align*}
Then another induction argument yields 
a sequence \aSeqilu{z_i}{i}{1}{\infty} of points of $\nP$
with the following properties
\begin{align}
z_i \in A_{n_i}
,
\qquad
\nM(x,z_1) < \nEps
,
\qquad
\nM(z_i,z_{i+1}) < \nEps \bDiv {b^{i}}
\qquad
\text{($i = 1,2,3,\ldots$)}
.
\label{eq:ml:zni:d}
\end{align}

The sequence \aSeq{z_i}{i}{1}{\infty} is Cauchy in $\nP$, 
because
\begin{align*}
\sum_{i=1}^{\infty} \nM(z_i,z_{i+1})
 \mEQx{\eqref{eq:ml:zni:d}}{\leq}
\sum_{i=1}^{\infty} \nEps \bDiv {b^{i}} =
\nEps \bDiv (b-1) < 
+\infty
.
\end{align*}
As the space $\nP$ is complete,
there is a point $y$ of $\nP$ such that $z_i \aLim{i} y$ in $\nP$.
Since all but a finite number of 
members of the sequence \aSeq{z_i}{i}{1}{\infty}
lie in the point set $\bigcup_{i = n}^{\infty} A_i$,
the limit $y$ of \aSeq{z_i}{i}{1}{\infty} 
belongs to its closure
$\overline{\bigcup_{i = n}^{\infty} A_i}$.
As this holds for any number $n$,
we then have $y \in A$ by definition.

Since $\nM(z_i,y) \aLim{i} 0$, 
there is an index $k$ such that $\nM(z_k,y) < \nEps \bDiv (b-1)$.
Then
\begin{align*}
\nM(x,y) & \leq
\nM(x,z_1) + \sum_{i=1}^{k-1} \nM(z_i,z_{i+1}) + \nM(z_k,y) 
 \mEQx{\eqref{eq:ml:zni:d}}{<}
\nEps + \sum_{i=1}^{k-1} \nEps \bDiv {b^{i}} + \nEps \bDiv (b-1)
\\ & <
\nEps + \nEps \bDiv (b-1) + \nEps \bDiv (b-1) = 
\nEps(b+1) \bDiv (b-1)
.
\end{align*}
Therefore,
\(
\nM(x,A) < \nEps(b+1) \bDiv (b-1)
\)
for every $b > 1$. 
Hence
\(
\nM(x,A) \leq \nEps
\).
\end{proof}


\begin{remark*}

The bound cannot be made tighter.
Indeed, consider the standard one-dimensional Euclidean space \aE{1} 
with metric $\nM(x,y) = |x-y|$.
As $1/n \aLim{n} 0$ in~\aE{1}, 
we have $\{1/n\} \aLim{n} \{0\}$ in \aEh{1},
and thus \aSeqilu{\{1/n\}}{n}{1}{\infty} is a Cauchy sequence in \aEh{1}. 
Furthermore, we also have
\begin{align*}
\bigcap_{n=1}^{\infty} \overline{\bigcup_{i=n}^{\infty} \{1/i\}} =
\bigcap_{n=1}^{\infty} \overline{\{1/i \mid i \geq n\}} =
\bigcap_{n=1}^{\infty} (\{1/i \mid i \geq n\} \cup \{0\}) =
\{0\}
.
\end{align*}
Clearly, 
$\nM(1,\{1/n\}) = 1-1/n < 1$ for $n = 1,2,3,\ldots$,
and $\nM(1,\{0\}) = 1$.

\end{remark*}


\begin{lemma}
\label{cl:ph}

$A \in \nPh$.

\end{lemma}

\begin{proof}
First we show that the point set $A$ is non-empty.
Since \aSeq{A_n}{n}{1}{\infty} is a Cauchy sequence, 
there is an index $m$ such that 
\begin{align*}
(\forall i > m) (\forall x \in A_m) (\nM(x,A_i) < 1)
.
\end{align*}
Fix $x \in A_m$
(an element such as this exists, as the set $A_m$ is non-void).
The point $x$ satisfies the assumptions of 
\hyperlink{cl:ml}{Main Lemma}
with $\nEps$ replaced by 1.
Hence $\nM(x,A) \leq 1$.
Since the distance $\nM(x,A)$ is finite,
the set $A$ has at least one element.

We now prove that the set $A$ is bounded.
As we have
\(
A \subseteq \overline{\bigcup_{i = 1}^{\infty} A_i}
\),
it suffices to show that the union $\bigcup_{i = 1}^{\infty} A_i$ is a bounded set.
Since \aSeq{A_n}{n}{1}{\infty} is a Cauchy sequence, 
there is an index $m$ such that 
\begin{align*}
(\forall i > m) (\forall x \in A_i) (\nM(x,A_m) < 1)
.
\end{align*}
Therefore,
\(
A_i \subseteq \{ x \in \nP \mid \nM(x,A_m) < 1 \}
\)
for every index $i > m$. 
Then
\begin{align*}
\bigcup_{i = 1}^{\infty} A_i =
\bigcup_{i = 1}^{m} A_i \cup \bigcup_{i = m+1}^{\infty} A_i \subseteq 
\bigcup_{i = 1}^{m} A_i \cup \{ x \in \nP \mid \nM(x,A_m) < 1 \}
.
\end{align*}
Consequently,
the point set $\bigcup_{i = 1}^{\infty} A_i$, 
being a subset of a finite union of bounded sets, 
is also bounded.

By definition, each set 
$\overline{\bigcup_{i = n}^{\infty} A_i}$ is closed.
Hence their intersection, 
the set $A$, is closed as well.
\end{proof}


\begin{lemma}
\label{cl:uaan}

$\nU(A,A_n) \aLim{n} 0$.

\end{lemma}

\begin{proof}
It is sufficient to show that
\begin{align*}
(\forall \nEps>0) (\exists m) (\forall n > m) (\forall x \in A) (\nM(x,A_n) < \nEps)
.
\end{align*}
Let $\nEps>0$.
Since \aSeq{A_n}{n}{1}{\infty} is a Cauchy sequence, 
there is a number $m$ such that
\begin{align}
(\forall i,j > m) (
\nU(A_i,A_j) 
< \nEps \bDiv 2
)
.
\label{eq:uaan:m}
\end{align}
Choose an index $n > m$ and a point $x$ of $A$.
By definition, we have 
$x \in \overline{\bigcup_{i = n}^{\infty} A_i}$.
Therefore,
there is an index $i \geq n$ and a point $y$ of $A_i$ 
such that $\nM(x,y) < \nEps \bDiv 2$.
From this we get
\begin{align*}
\nM(x,A_n) \leq 
\nM(x,y)+\nM(y,A_n) \leq 
\nM(x,y)+\nU(A_i,A_n) \mEQx{\eqref{eq:uaan:m}}{<} 
\nEps \bDiv 2 + \nEps \bDiv 2 = 
\nEps
.
\tag*{\qedhere}
\end{align*}
\end{proof}


\begin{lemma}
\label{cl:uana}

$\nU(A_n,A) \aLim{n} 0$.

\end{lemma}

\begin{proof}
It suffices to prove that
\begin{align*}
(\forall \nEps>0) (\exists m) (\forall n > m) (\forall x \in A_n) (\nM(x,A) \leq \nEps)
.
\end{align*}
Let $\nEps>0$.
Since \aSeq{A_n}{n}{1}{\infty} is a Cauchy sequence, 
there is a number $m$ such that
\begin{align}
(\forall i,j > m) (
\nU(A_i,A_j) 
< \nEps)
.
\label{eq:uana:m}
\end{align}
Pick an index $n > m$ and a point $x$ of $A_n$.
From \eqref{eq:uana:m} it follows that
\[
(\forall j > n) (
\nM(x,A_j) 
< \nEps)
.
\]
The point $x$ satisfies the assumptions of 
\hyperlink{cl:ml}{Main Lemma}.
Hence $\nM(x,A) \leq \nEps$.
\end{proof}


\section*{Conclusion}

There are three intensionally different, albeit equivalent, 
characterizations of limits in Hausdorff hyperspaces.
The first one is described by \eqref{eq:limit},
the other two are based on the notions of 
the lower and upper limit of sequences of point sets.
These are defined as follows:
\begin{itemize}
\item
The point $x$ belongs to the lower limit of \aSeq{A_n}{n}{1}{\infty},
in symbols $x \in \nLiminf A_n$,
if there is a sequence $\aSeq{y_n}{n}{1}{\infty}$ such that 
$y_n \in A_n$ and $y_n \aLim{n} x$.
\item
The point $x$ belongs to the upper limit of \aSeq{A_n}{n}{1}{\infty},
in symbols $x \in \nLimsup A_n$,
if there is a subsequence \aSeq{A_{n_i}}{i}{1}{\infty} and 
sequence $\aSeq{y_i}{i}{1}{\infty}$ such that 
$y_i \in A_{n_i}$ and $y_i \aLim{i} x$.
\end{itemize}
If the sequence of point sets \aSeq{A_n}{n}{1}{\infty} has the limit $A$ in the Hausdorff hyperspace, 
then proceeding analogously as in
\cite[17.6.1--17.6.3]{bib:Cech}
we obtain%
\footnote{
The last equality holds for any (not necessarily convergent) sequence of point sets.
}
\begin{align*}
A =
\nLiminf A_n =
\nLimsup A_n = 
\bigcap_{n=1}^{\infty} \overline{\bigcup_{i=n}^{\infty} A_i} 
.
\end{align*}
It is a matter of preference, which characterization is used.
However, completeness proofs based on lower or upper limits
tend to be quite technical, as they mainly rely 
on several constructions of point sequences in the base space.

For this reason, we prefer the characterization based on \eqref{eq:limit}.
It has been used, for instance, in the proof of Theorem~17.6.5 in
\cite{bib:Cech}---the inspiration for this note---%
for hyperspaces consisting of non-empty compact point sets.
Their Theorem~15.7.2 has a similar role as our
\hyperlink{cl:ml}{Main Lemma} 
for us;
however, the proof of this auxiliary claim and its application is 
still quite complex and not very insightful.
The appeal of 
\hyperlink{cl:ml}{Main Lemma}
is that it is phrased solely in terms of neighborhoods
with absolutely no reference to the used construction by induction of a point sequence---%
the only one in our completeness proof.
Once established, the remaining lemmas are followed easily.



\end{document}